\documentclass[reqno,11pt]{article}
\usepackage{a4wide,xcolor,eucal,enumerate,mathrsfs}
\usepackage[normalem]{ulem}
\usepackage{pdfsync}
\usepackage[hidelinks]{hyperref}
\usepackage{amsmath,amssymb,epsfig,amsthm}
\usepackage[latin1]{inputenc}
\usepackage[english]{babel}

\numberwithin{equation}{section}

\theoremstyle{plain}
  \newtheorem{theorem}{Theorem}
  
  \newtheorem{lemma}[theorem]{Lemma}
  \newtheorem{corollary}[theorem]{Corollary}
\theoremstyle{definition}
  
\theoremstyle{remark}

\newcommand{\comma}{\textrm{,}}
\newcommand{\period}{\textrm{.}}
\newcommand{\bra}[1]{\left( #1 \right)}
\newcommand{\sqa}[1]{\left[ #1 \right]}
\newcommand{\cur}[1]{\left\{ #1 \right\}}
\newcommand{\ang}[1]{\left< #1 \right>}
\newcommand{\abs}[1]{\left| #1 \right|}
\newcommand{\nor}[1]{\left\| #1 \right\|}

\newcommand{\rnum}{\mathbb{R}}
\newcommand{\nnum}{\mathbb{N}}

\newcommand{\veps}{\varepsilon}

\newcommand{\cF}{\mathcal{F}}

\newcommand{\eE}{\mathbb{E}}

\makeindex             

\begin{document}

\title{A short proof of Stein's universal multiplier theorem}
\author{Dario Trevisan\
\thanks{Scuola Normale Superiore, Pisa, \textsf{dario.trevisan@sns.it}}
}
%
\maketitle

\abstract{We give a short proof of Stein's universal multiplier theorem, purely by probabilistic methods, thus avoiding any use of harmonic analysis techniques (complex interpolation or transference methods).}

\section{Introduction}

The celebrated Stein's universal multiplier theorem \cite[Corollary~IV.6.3]{stein} provides strong $(L^p, L^p)$-bounds for a general family of operators related to a Markovian semigroup $(T^t)_{t\ge0}$, virtually without any assumption on the underlying measure space $(X,m)$. More recent proofs of this classical result are based on analytic methods (see \cite{MR0500219, MR639463} and the monograph \cite{MR0481928}), which also shows that the Markovianity assumption on the semigroup can be removed, keeping only the $L^p$-contractivity assumption. On the other hand, Stein's original proof relies on deep connections with martingale theory; not much later, P.A.\ Meyer began to investigate the problem purely by stochastic methods (see e.g.\ \cite{MR0501379} and subsequent articles, and also \cite{MR889471} for an exposition of the transference approach).

In \cite{stein}, the multiplier theorem is actually a corollary of $L^p$-bounds for suitable Littlewood-Paley $g$-functions, which follows from a clever complex interpolation between the $L^2$ case, which holds by spectral theory, and an $L^p$-inequality, obtained by martingale tools. From a probabilist's viewpoint, this interpolation argument could be a mountain to climb:  Meyer literally wrote that \emph{on ne ``comprend'' pas ce qui se passe}  \cite[end of Section~1]{MR0501379}.

In this note we prove the multiplier theorem (Theorem~\ref{theo-1} below) relying only on martingale tools, namely Rota's construction and Burkh\"older-Gundy inequalities, the main contribution being therefore that we avoid the use of complex interpolation. With hindsight, this result could be considered as an analogue of the short proof of the maximal theorem for Markovian semigroups, sketched in \cite[below Theorem~IV.4.9]{stein}: like in that case, powerful analytical tools can prove results for rather general semigroups but, in the Markovian setting, probability is enough and gives much simpler proofs. It is remarkable that, apparently, this shortcut went unnoticed, maybe because both Stein and Meyer were focusing  mainly on Littlewood-Paley functions.

The proof allows also for an easy computation of the constants involved (in terms of $p$) and also for the norm of operators given by imaginary powers of the generator of the semigroup \cite[Corollary IV.6.4]{stein}. Assuming these bounds only, one can then deduce boundedness of  $g$-functions \cite[Theorem~1.1]{MR1264824}. Another application (Corollary \ref{coro-1}) comes from the fact that some form of Burkh\"older-Gundy inequalities still holds true for $p=1$ (Davis' Theorem): we remark that one might also obtain analog results in a general, non-Markovian, setting by extrapolation on the $L^p$ bounds (for a detailed account on Yano's extrapolation theory, see \cite{MR2122269}).

After writing this note, we discovered that a similar argument already appeared in the last section of \cite{shigekawa}: there, however, continuous-time stochastic calculus is widely used and it is not fully recognized that Rota's construction and Burkh\"older-Gundy inequalities suffice, without any assumption on the underlying measure space.

\section{Setting}

We briefly recall the setting and notation of \cite[Chapters~III and IV]{stein}. Let $(X, dx)$ be a $\sigma$-finite measure space and let $\bra{T^t}_{t\ge0}$ be a strongly continuous semigroup of operators defined on $L^2\bra{X,dx}$ such that the following conditions hold:
\begin{enumerate}
\item $\nor{T^tf}_p \le \nor{f}_p$ ($1\le p \le \infty$) (contraction);
\item $T^t$ is self-adjoint on $L^2\bra{X,\mu}$, for every $t\ge0$ (symmetry);
\item $T^tf \ge 0$ if $f\ge0$ (positivity);
\item $T^t 1 = 1$ (conservation of mass),
\end{enumerate}
where $T^t 1$ is defined by $\sup_n T^t I_{A_n}$, taking a sequence $A_n \uparrow X$, with $I_{A_n} \in L^2$ (this is well defined in general because of positivity). The infinitesimal generator $A$ of $\bra{T^t}_{t\ge0}$ in $L^2\bra{X,dx}$ is given by
\[ A = \lim_{t\downarrow 0} \frac{f - T^t f}{t}\comma \]
for any $f\in L^2\bra{X,\mu}$, whenever the limit exists in $L^2\bra{X,dx}$ (that defines its domain $D(A)$).

Because of symmetry and contraction assumptions on $T^t$, the generator $A$ is a self-adjoint, non-negative and densely defined operator. By spectral theory, there exists a unique resolution of the identity $\bra{E\bra{\lambda}}_{\lambda \in \rnum}$ associated to $A$. In particular, the representation of $T^t = e^{-tA}$, $t\ge 0$, holds in the following sense: 
\begin{equation} \label{eq-Tt} \ang{T^t f, g} = \int_0^\infty e^{-t\lambda} d \ang{E\bra{\lambda} f, g} \comma \end{equation}
where, for any $f$, $g \in L^2\bra{X,dx}$, $\lambda\mapsto \ang{E\bra{\lambda} f, g}$ is a bounded variation function on $\rnum$, with total variation not greater than $\nor{f}_2\nor{g}_2$.

On the other side, positivity and conservation assumptions allows for a dynamical realization of the semigroup as the transition semigroup associated to a Markov process, with state space $X$ and $dx$ as invariant measure: this is the content of \cite[Theorem~IV.4.9]{stein} (due to G.C.~Rota), that we describe here in a more explicit form (and actually a bit simplified, as we require only a finite product space). Note that, having no assumption on $X$, it is not clear whether $T^\veps$ is induced by some probability kernel; still, the proof proceeds as in the case of existence of Markov chains.

Given $\veps>0$ and $N \in \nnum$, let $\Omega = X^{N+1}$, endowed with the product $\sigma$-algebra. For $k \in \cur{0,\ldots,N}$, let $\pi_k$ be the projection on the $k$-th factor, let
\[ \cF_{k} = \sigma\bra{\pi_k, \pi_{k+1}, \ldots, \pi_N} \]
which defines an reverse (i.e.\ decreasing) filtration and let $\hat{\cF} = \sigma\bra{\pi_0}$. Then, there exists a $\sigma$-finite measure $\mathbb{P} = \mathbb{P}_{\veps, n}$ on $\Omega$ such that the law of $\pi_0$ w.r.t.\ $\mathbb{P}$ is $dx$ and for $k \in \cur{0,\ldots,N}$, $\mathbb{P}$ is $\sigma$-finite on $\cF_k$ and for every $f \in L^1\bra{X,dx}$, it holds
\[ T^{\veps } f \bra{ \pi_0 } =  \hat{\eE}\sqa{ \eE_k \sqa{f\bra{\pi_0}} }  = \hat{\eE}\sqa{ f_k} \comma\]
where $\hat{\eE}$ denotes the conditional expectation operator w.r.t.\ $\hat{\cF}$, $\eE_k$ is the same, w.r.t.\ $\cF_k$ and $f_k = \eE_k\sqa{f\circ \pi_0}$ is a reverse martingale: of course, being the set of times finite, there is no problem in applying the usual theory of martingales. Moreover, it is not difficult to check that all the properties and theorems used here and in what follows, which are well known to hold in probability spaces, extend verbatim to the $\sigma$-finite case, the only exception being Corollary \ref{coro-1} below.

We recall now the special case of spectral multipliers problem, addressed by Stein. Let $M$ be a bounded Borel function on $(0,\infty)$ and define, for $\lambda >0$,
\begin{equation}\label{eq-m}
 m\bra{\lambda} = - \lambda \int_0^\infty M\bra{t} e^{-t \lambda } dt\comma\end{equation}
(let also $m\bra{0} = 0$), which is a so-called multiplier of Laplace transform type, when we use it to define, by means of spectral calculus, the operator
\begin{equation}\label{eq-Tm} T_m f = \int_0^\infty m\bra{\lambda} dE\bra{\lambda}f \comma\end{equation}
for $f \in L^2\bra{X,dx}$. Since 
\begin{equation}\label{eq-supm} \nor{m}_{\infty} = \sup_\lambda \abs{m\bra{\lambda}} \le \sup_t \abs{M\bra{t}} = \nor{M}_\infty \comma\end{equation}
it follows by the spectral theorem that $T_m$ is well defined and maps continuously $L^2\bra{X,dx}$ into itself, with operator norm $\nor{T_m}_{2,2} \le  \nor{M}_\infty$. The problem consists in proving that, for $p \in ]1,\infty[$, $T_m$ maps continuously $L^2\cap L^p \bra{X,dx}$ into itself.

\section{Proof of the multiplier theorem}

We are in a position to state and prove Stein's result.

\begin{theorem}[Stein's multiplier theorem]\label{theo-1}
Let $p \in ]1,\infty[$. Then, $T_m$ is a bounded linear operator on $L^2\cap L^p\bra{X,dx}$, with \[ \nor{T_m}_{p,p} \le C_p \nor{M}_\infty \comma \]
where $C_p = O(\bra{p-1}^{-1})$ as $p \downarrow 1$.
\end{theorem}

We sketch heuristically the line of reasoning. By substituting \eqref{eq-m}, which gives $m$ in terms of $M$, into \eqref{eq-Tm} and exchanging integrals, we obtain the expression
\begin{equation}\label{eq:T_m}
T_m f = \int_0^\infty M\bra{t} \sqa{ \int_0^\infty -\lambda e^{-t\lambda} dE\bra{\lambda}f } dt = \int_0^\infty M\bra{t} \frac{d}{dt} T^tf dt \comma\end{equation}
where we also used \eqref{eq-Tt}. Then, we formally simplify the increments $dt$ and recall that, by Rota's construction, it holds $T^t f = \hat{\eE} \sqa{  f_t }$, where $f_t$ is some reverse martingale:
\[ T_m f = \int_0^\infty M\bra{t} d T^t f = \hat{\eE} \sqa{ \int_0^\infty M\bra{t} d f_t }\period\]
To estimate the $L^p$ norm, we use the fact that $ \hat{\eE}$ is a contraction and Burkh\"older-Gundy inequalities, obtaining
\[ \nor{T_mf}_p \le C_p \nor{M}_\infty \nor{f}_p\period \]

To make this reasoning rigorous, we first consider the case when $M$ is a step function and then we pass to the limit. To do this, we state and prove two elementary lemmas, the first being in fact a special case of \eqref{eq:T_m} for step functions.

\begin{lemma}\label{lemma-1}
Given $N \in \nnum$, let $0 = t_0 \le t_1 \le \ldots \le t_N < \infty$ and let 
\begin{equation} \label{M-step} M = \sum_{i=0}^{N-1} M_i I_{[t_{i}, t_{i+1}[} \period \end{equation}
Then, for every $f \in L^2\bra{X,dx}$ it holds
\begin{equation}\label{Tm-step}  T_m f = \sum_{i=0}^N M_i \bra{T^{t_{i+1}} f - T^{t_i} f}  \period\end{equation}
\end{lemma}

\begin{proof}
Since $M \mapsto m \mapsto T_m$ is linear, it is enough to consider the case $M = I_{[0,t[}$ and prove that $T_m = T^t- Id$. Integrating by parts, we have $m\bra{\lambda} = e^{-t\lambda} - 1$ and so we conclude by \eqref{eq-Tt}.\qed
\end{proof}

\begin{lemma}\label{lemma-2}
Let  $\bra{M^n}_{n\ge0}$ be a sequence of Borel functions, with $\nor{M_n}_\infty$ uniformly bounded and converging $\mathscr{L}^1$-a.e.\ to some function $M$. Then, for every $f \in L^2\bra{X,dx}$, it holds 
\[ \lim_{n\to \infty} T_{n}f  = T_m f \quad \textrm{ in $L^2\bra{X,dx}$,} \]
where $T_{n}$ denotes the operator defined by $M^n$ in place of $M$.
\end{lemma}

\begin{proof}
As above, by linearity, it is enough to consider the case $M =0$ (i.e.\ $m=0$ and $T_m = 0$). By dominated convergence, from \eqref{eq-m} we obtain that, for every $\lambda\in [0,\infty[$, $\abs{m_n\bra{\lambda}}$ converges to zero. From \eqref{eq-supm} and the assumption on $\bra{M^n}_{n\ge0}$ this convergence is dominated by some constant, and this suffices to pass to the limit. Indeed, given $f \in L^2\bra{X,dx}$, by spectral theorem, it holds
\[ \nor{T_{n}f }_2^2  = \ang{ T_{n}f, T_{n} f }  = \int_0^\infty \abs{m_n\bra{\lambda}}^2 d \ang{E\bra{\lambda}f, f} \period\]
As already remarked, $d\ang{E\bra{\lambda}f, f}$ is a finite positive measure and so we conclude by dominated convergence.\qed
\end{proof}

\emph{Proof of Theorem \ref{theo-1}.}
We may assume that $\abs{M} \le 1$. First, let $M$ be a step function of the form \eqref{M-step}, where $\veps = t_{i+1}-t_i$ constant for $i \in \cur{0,\ldots, N-1}$. If we apply Rota's theorem, as described in the previous section, we obtain from Lemma \ref{lemma-1} above that 
\[  T_m f\circ \pi_0  =  \sum_{i=0}^{N-1} M_i \bra{\hat{\eE}\sqa{ f_{i+1} } - \hat{\eE}\sqa{f_i}} =  \hat{\eE}\sqa{ \sum_{i=0}^{N-1} M_i \bra{  f_{i+1} - f_i } }, \quad \text{$\mathbb{P}$-a.s.~in $\Omega$.}\]
It holds therefore
\[ \nor{T_mf}_{p} = \nor{\hat{\eE}\sqa{ \sum_{i=0}^{N-1} M_i  \bra{ f_{i+1} - f_i } } }_{p}\comma \]
where the first norm is computed in $L^p\bra{X,dx}$ and the other in $L^p\bra{\Omega, \mathbb{P}}$, because the law of $\pi_0$ is $dx$. Since conditional expectations are contractions, we have
\[ \nor{T_m f}_p \le \nor{ \sum_{i=0}^N M_i  \bra{f_{i+1} - f_i} }_{p} \period\]
We apply Burkh\"older-Gundy inequalities for martingale transforms (e.g.\ \cite[Theorem~IV.4.2]{stein}) to the reverse martingale above:
\[ \nor{ \sum_{i=0}^N M_i  \bra{f_{i+1} - f_i} }_p \le c_p \nor{f}_p \comma \]
where $c\bra{p}$ is a constant, depending only on $p$: the claimed bound for $p \downarrow 1$ follows from constants-chasing in Marcinkiewicz interpolation.
In the general case, we approximate a $M$ with a sequence of step functions $\bra{M^n}_{n\ge1}$ such that $\abs{M^n} \le 1$ for every $n$ and $M^n \bra{s} \to M\bra{s}$, $\mathscr{L}^1$-a.e.\ $s \in [0,\infty[$: this possibility is well-known, as it follows e.g.\ by density in $L^1\bra{[0,\infty[, \mathscr{L}^1}$ of step functions and a diagonal argument. For a fixed $f\in L^2\cap L^p\bra{X,dx}$, Lemma \ref{lemma-1} entails that, up to a subsequence, $(T_{n}f)$ converge $dx$-a.e.\ to $T_m f$. By Fatou's lemma, it holds
\[ \nor{T_m f}_p \le \liminf_{n\to \infty} \nor{T_{n} f }_p \le c_p \nor{f}_p,\]
that gives the thesis. \qed

\begin{corollary}\label{coro-1}
If $\bra{X,dx}$ has finite measure $\abs{X}$, it holds for every $f \in L^2\bra{X,dx}$,
\[ \nor{T_m f}_1 \le c\abs{X} \nor{M}_\infty \nor{ f}_{L \log L} \comma\]
where $c>0$ is some universal constant.
\end{corollary}

\begin{proof}
We may assume that $\abs{X} =1$ and $\nor{M}_\infty \le 1$. Arguing as above, we apply Davis' Theorem instead of Burkh\"older-Gundy inequalities, e.g.\ \cite[Theorem~2.1]{MR580107}:
\[\eE \sqa{ \abs{ \sum_{i=0}^N M_i  \bra{f_{i+1} - f_i} } } \le c \eE\sqa{ \bra{\sum_{i=0}^{N-1} \abs{f_{i+1} - f_i }^2} ^{1/2} } \le c  \eE \sqa{  \sup_{i=0,\ldots N}\abs{ f_i }  } \period\]
Then, we use the well-known corollary of Doob's inequality, that allows to control the $L^1$ norm of a maximal function of martingale (closed by $f$) in terms of the $L\log L$ norm of $f$.\qed
\end{proof}

\noindent {\bf Acknowledgements.} The author is member of the GNAMPA group of the Istituto Nazionale di Alta Matematica (INdAM). He also thanks G.M.\ Dall'Ara for many discussions on the subject.

\providecommand{\bysame}{\leavevmode\hbox to3em{\hrulefill}\thinspace}
\providecommand{\MR}{\relax\ifhmode\unskip\space\fi MR }
\providecommand{\MRhref}[2]{%
  \href{http://www.ams.org/mathscinet-getitem?mr=#1}{#2}
}
\providecommand{\href}[2]{#2}


\begin{thebibliography}{CRW78}

\bibitem[Cow81]{MR639463}
M.~G.~Cowling, \emph{On {L}ittlewood-{P}aley-{S}tein theory}, Proceedings
  of the {S}eminar on {H}armonic {A}nalysis ({P}isa, 1980), 1981.

\bibitem[CRW78]{MR0500219}
R.~R. Coifman, R.~Rochberg, and G.~Weiss, \emph{Applications of
  transference: the {$L\sp{p}$} version of von {N}eumann's inequality and the
  {L}ittlewood-{P}aley-{S}tein theory}, Linear spaces and approximation, Birkh\"auser,
  Basel, 1978.

\bibitem[CW76]{MR0481928}
R.~R.~Coifman and G.~Weiss, \emph{Transference methods in analysis},
  American Mathematical Society, Providence, R.I., 1976.

\bibitem[KM05]{MR2122269}
G.~E. Karadzhov and M.~Milman, \emph{Extrapolation theory: new results and
  applications}, J. Approx. Theory \textbf{133} (2005), no.~1, 38--99.

\bibitem[LLP80]{MR580107}
E.~Lenglart, D.~L{\'e}pingle, and M.~Pratelli, \emph{Pr\'esentation unifi\'ee
  de certaines in\'egalit\'es de la th\'eorie des martingales}, S\'eminaire de
  {P}robabilit\'es, {XIV} , Lecture Notes in Math.,
  vol. 784, Springer, Berlin, 1980.

\bibitem[Med95]{MR1264824}
S.~Meda, \emph{On the {L}ittlewood-{P}aley-{S}tein {$g$}-function}, Trans.
  Amer. Math. Soc. \textbf{347} (1995), no.~6, 2201--2212.

\bibitem[Mey76]{MR0501379}
P.~A.~Meyer, \emph{D\'emonstration probabiliste de certaines in\'egalit\'es de
  {L}ittlewood-{P}aley. {I}. {L}es in\'egalit\'es classiques}, S\'eminaire de
  {P}robabilit\'es, {X}, Lecture Notes in Math., Vol. 511, Springer, Berlin, 1976,
  pp.~125--141.

\bibitem[Mey85]{MR889471}
P.-A.~Meyer, \emph{Sur la th\'eorie de {L}ittlewood-{P}aley-{S}tein (d'apr\`es
  {C}oifman-{R}ochberg-{W}eiss et {C}owling)}, S\'eminaire de probabilit\'es,
  {XIX},  Lecture Notes in Math., vol.~1123, Springer, Berlin, 1985,
  pp.~113--129.

\bibitem[Shi97]{shigekawa}
I.~Shigekawa, \emph{The {M}eyer inequality for the {O}rnstein-{U}hlenbeck
  operator in {$L\sp 1$} and probabilistic proof of {S}tein's {$L\sp p$}
  multiplier theorem}, Trends in probability and related analysis ({T}aipei,
  1996), World Sci. Publ., River Edge, NJ, 1997, pp.~273--288.

\bibitem[Ste70]{stein}
E.M.~Stein, \emph{Topics in harmonic analysis related to the
  {L}ittlewood-{P}aley theory.}, Annals of Mathematics Studies, No.~63,
  Princeton University Press, Princeton, N.J., 1970.
  
\end{thebibliography}
\end{document}